\documentclass[11pt]{amsart}
\usepackage{amssymb,amsmath,eucal,amscd,amsthm,hyperref}
\usepackage{paralist}

\newtheorem{theorem}{Theorem}[section]

\newtheorem{lemma}[theorem]{Lemma}

\newtheorem{corollary}[theorem]{Corollary}
\newtheorem{conjecture}[theorem]{Conjecture}
\newtheorem{question}[theorem]{Question}

\theoremstyle{definition}
\newtheorem{definition}[theorem]{Definition}

\theoremstyle{remark}
\newtheorem{remark}[theorem]{Remark}

 \def\sl{\operatorname{SL}}  
\def\ll{\langle} \def\rr{\rangle}
 \def\a{\alpha}   \def\bp{\begin{pmatrix}}
 \def\ep{\end{pmatrix}}
\def\bn{\begin{enumerate}}
\def\bnc{\begin{compactitem}} \def\enc{\end{compactitem}}

   \def\en{\end{enumerate}}
\def\ba{\begin{array}} \def\ea{\end{array}}  \def\Si{\Sigma}  \def\a{\alpha}

\def\ker{\operatorname{Ker}}\def\be{\begin{equation}} \def\ee{\end{equation}}

\def\co{\colon} 
\def\PP{\mathcal{P}}
\def\Sol{\mathrm{Sol}}
\def\Nil{\mathrm{Nil}}
\newcommand{\C}{\mathbb{C}}

\newcommand{\F}{\mathbb{F}}
\newcommand{\N}{\mathbb{N}}

\newcommand{\R}{\mathbb{R}}
\newcommand{\Z}{\mathbb{Z}}

\title[Decision problems for $3$-manifolds]{Decision problems for $3$-manifolds and their fundamental groups}
\author{Matthias Aschenbrenner}
\address{University of California, Los Angeles, California, USA}
\email{matthias@math.ucla.edu}

\author{Stefan Friedl}
\address{Fakult\"at f\"ur Mathematik\\ Universit\"at Regensburg\\   Germany}
\email{sfriedl@gmail.com}

\author{Henry Wilton}
\address{DPMMS\\ CMS\\ Wilberforce Road\\ Cambridge\\ CB3 0WB\\ UK}
\email{h.wilton@maths.cam.ac.uk}

\begin{document}

\maketitle

\begin{abstract}
We survey the status of some decision problems for $3$-mani\-folds and their fundamental groups.
This includes the classical decision problems for finitely presented groups (Word Problem, Conjugacy Problem, Isomorphism Problem), and also the Homeomorphism Problem for $3$-manifolds and the Membership Problem for
$3$-manifold groups.
\end{abstract}

\section*{Introduction}

\noindent
The classical group-theoretic decision problems were formulated by Max Dehn in his work on the topology of surfaces \cite{dehn_unendliche_1911} about a century ago.  He considered the following questions about finite presentations $\langle A\mid R\rangle$ for a group $\pi$:
\bnc
\item[1.] the {\it Word Problem,}\/ which asks for an algorithm to determine whether  a word on the generators $A$ represents the identity element of~$\pi$;
\item[2.] the {\it Conjugacy Problem,}\/ which asks for an algorithm to determine whether  two words on the generators $A$ represent conjugate elements of~$\pi$;
\item[3.]  the {\it Isomorphism Problem,}\/ which asks for an algorithm to determine whether  two given finite presentations represent isomorphic groups.
\enc
Viewing $\pi$ as the fundamental group of a topological space (represented as a simplicial complex, say), these questions can be thought of as asking for algorithms to determine whether a given loop is null-homotopic, whether a given pair of loops is freely homotopic, and whether two aspherical spaces are homotopy-equivalent, respectively.  We   add some further questions that arise naturally:
\bnc
\item[4.] the {\it Homeomorphism Problem,}\/ which asks for an algorithm to determine whether two given triangulated manifolds are homeomorphic;  
\item[5.] the {\it Membership Problem,}\/ where the goal is to determine whether a given element of a group lies in a specified subgroup.
\enc
Since the  1950s, it has been known that problems (1.)--(5.) have negative answers in the generality in which they  have been formulated above. A finitely presented group with undecidable  Conjugacy Problem was first constructed in 1954 by
P.~S.~Novikov~\cite{No54}, and soon thereafter he \cite{No55} and W.~Boone \cite{boone} independently found 
a finitely presented group with undecidable Word Problem. 
(See  \cite{stillwell} for an exposition.)
Hence Problem~(5.), being a generalization of (1.),
is also undecidable. 
Similarly, Problem~(3.) is undecidable, since Adyan~\cite{adyan57a, adyan57b} and Rabin~\cite{rabin} showed that
there is no algorithm for the more restrictive problem of determining whether a given finite presentation describes the trivial group.
Problem (4.) is undecidable even if we restrict to smooth manifolds
of the same given dimension~$\geq 4$; this was shown by Mar\-kov~\cite{markov58}, as a corollary of the unsolvability of~(3.).

In contrast to this, in this paper we show that all these problems can now be solved for compact $3$-manifolds and their fundamental groups, with the caveats that in (3.)~we restrict ourselves to closed, orientable $3$-manifolds, and in (4.)~we restrict ourselves to orientable, irreducible $3$-manifolds.  Contrary to common perception (see, e.g., \cite[Section~7.1]{poonen})
 the Homeomorphism Problem for $3$-manifolds is still open for reducible $3$-manifolds. We  discuss the status of the Homeomorphism Problem in detail in Section~\ref{section:homeomorphismproblem}.

The solutions to the first four problems, with the aforementioned caveats, have been known to the experts for a while, and as is to be expected, they all rely on the Geometrization Theorem.
An algorithm for solving the  Membership Problem was recently given in \cite{FW14},
and we will provide a summary of the main ideas in this paper. The solution to the Membership Problem requires not only the Geometrization Theorem
 but also the Tameness Theorem of Agol~\cite{Ag04} and Calegari--Gabai~\cite{CG06},  the 
  Virtually Compact Special Theorem
 due to Agol~\cite{Ag13} and Wise~\cite{Wi09,Wi12a,Wi12b}, as well as a result of  Kapovich--Miasnikov--Weidmann~\cite{kapovich_foldings_2005}.

 Once a decision problem has been shown to be solvable, a natural next question concerns its complexity, and its
  implementability.
The complexity of decision problems around $3$-manifolds is a fascinating topic which
is not touched upon in the present paper, except for a few references to the literature here and there.
Practical implementations of algorithms for $3$-manifolds are discussed, e.g., in \cite{Mat03}.

The paper is organized as follows. In Section~\ref{section:decisionproblems} we   recall the precise statement of the (uniform) decision problems.
In Section~\ref{section:basicdefinitions} we   collect 
some basic definitions in group theory and $3$-manifold topology.
Then, we    recall in Section~\ref{section:3manifolds} some of the key results in the study of $3$-manifolds which   appear time and again in the solutions to the decision problems. Finally in Section~\ref{section:solutions} we  show that the five aforementioned problems are indeed solvable for (most) $3$-manifolds and their fundamental groups. We conclude with a list of open problems in Section~\ref{sec:open problems}.

A final remark: the paper is written for group theorists and low-dimen\-sional topologists alike. We hope that we succeeded in striking a balance between giving an exposition at an appropriate level of precision and conveying the main ideas behind the arguments. However, in order to make the paper more readable we refrained from giving some arguments in full detail.

\subsection*{Convention.} All $3$-manifolds are assumed to be compact
 and connected. All surfaces in a $3$-manifold are  compact and properly embedded. 

\subsection*{Acknowledgment.}
The first author was partially supported by NSF grant DMS-0969642. The second author is grateful to the organizers of the conference `Interactions between low dimensional topology and mapping class groups' at the Max-Planck Institute for Mathematics in Bonn for giving him an opportunity to speak. The second author was supported by the CRC `higher invariants' at the University of Regensburg. The third author was partially supported by an EPSRC Career Acceleration Fellowship.
We wish to thank Benjamin Burton, David Futer and Greg Kuperberg for helpful comments. We are also very grateful for the detailed feedback provided by the referees.

\section{Decision Problems}\label{section:decisionproblems}

\noindent
To formulate group-theoretic decisions problems such as those from the introduction rigorously (and thus, to be able to prove their  undecidability in general),
one needs to make precise the informal notion of ``algorithm.''

\subsection{Turing machines}
Several mathematically rigorous definitions of algorithm were proposed in the 1930s, one of them being the notion of a {\it Turing machine,} defined by Turing in~\cite{turing}.
They were subsequently shown to all be equivalent  in a natural sense. This is usually seen as evidence for the {\it Church-Turing Thesis,}\/ 
which asserts that {\it every}\/ effective computation (be it on an abacus or a modern-day computer) can be carried out by a Turing machine. Hence
in practice, to show  the (mechanical) {\it decidability}\/ of a mathematical question, one usually just sketches an algorithm {\it informally}\/ and then  appeals to the Church-Turing Thesis
to ascertain that it could  in principle be translated into an actual Turing machine if so desired.
In contrast, to prove {\it undecidability}\/ of a given problem, one needs to resort to the precise definition of one's preferred model of computation. Since in this paper, we are mainly concerned with
decidability results, below we only give a rough description of the structure and workings of a Turing machine; for a detailed and mathematically rigorous discussion see, e.g.,~\cite{HU}.

A {\it Turing machine}\/ is a hypothetical device consisting of a finite-length program as well as
\begin{enumerate}
\item an infinite memory, usually visualized as a {\it tape}\/ divided into infinitely many {\it cells,}\/ sequentially ordered, where each cell can store exactly one symbol, either $0$ or $1$, or be blank;
\item a {\it head}\/ that can read, write and delete
 symbols on the tape and move the tape left and right one (and only one) cell at a time; and 
\item a {\it state register}\/ that stores the current state of the Turing machine, one of finitely many. 
\end{enumerate}
The Turing machine is provided with some {\it input}\/ on the tape, with all but finitely many of the cells of the tape being initially blank.
It then proceeds to modify this input according to its predetermined program, until it reaches a distinguished {\it terminal state.}\/ If the machine 
never reaches this state, the computation will go on forever; but if it does (we say: if the Turing machine {\it terminates}\/), then the non-blank
cells left on the tape are called the {\it output}\/ of the Turing machine.

Turing machines can be viewed as accepting other objects  as input and computing other objects as output, provided an encoding of these objects into finitely many finite strings of $0$s and $1$s has been chosen; these strings can then be written onto the tape of the Turing machine, separated by blanks, to provide the input for the machine. For example, 
Turing machines can work with words in the usual Latin alphabet, since its letters   can be encoded by the binary representations of their ASCII codes. 
In fact, we may just as well use any countable, potentially infinite alphabet.
More relevant for this paper, a finite presentation $\langle A\,|\,R\rangle$ of a group can
be fed into a Turing machine by listing first an encoding of the elements of the generating set $A$   in binary representation, followed by  the words in $R$, again suitably encoded into natural numbers via binary representation, and by separating these binary strings  by blanks. Similarly one can also treat other finite mathematical objects, like finite simplicial complexes.
Indeed,  the (finite) data describing a Turing machine itself can be encoded (by a natural number) and thus be fed as an input to a Turing machine. A {\it universal Turing machine}\/ is one that takes as input such an encoding of a Turing machine $T$, together with some input  $s$, and then simulates the computation of $T$ with input $s$. In essence (except for the lack of infinitely large memory and the possibility of computational errors due to hardware failure), modern computers are implementations of universal Turing machines; cf.~\cite{universalTM}.

A {\it decision problem}\/ is a subset of the set of possible inputs for a Turing machine. One says that a Turing machine {\it solves}\/ a decision problem $S$ if it terminates for each possible input, {\it and}\/ if it outputs ``YES'' if the input belonged to $S$, and ``NO'' otherwise.
A decision problem is {\it solvable}\/ if there exists a Turing machine solving it.
Since there are only countably many Turing machines, but uncountable many decision problems, it is clear that there are many unsolvable decision problems.
Using the concept of universal Turing machine, Turing~\cite{turing} first  exhibited a particular unsolvable decision problem, the {\it Halting Problem}\/: there is no Turing machine which takes as input an encoding of a Turing machine $T$
  and decides whether $T$ halts upon input of the empty tape (with all cells blank).
See \cite{poonen} for a survey of many naturally occurring decision problems in mathematics which turned out to be undecidable.

A set $S$ of possible outputs of a Turing machine is called {\it recursively enumerable}\/ if 
there exists a Turing machine, possibly non-terminating, that prints the elements of $S$, separated by blanks, onto its tape. Here we allow the possibility of repetitions, but it is easy to see that this notion of ``recursively enumerable'' does not change if we insist that the Turing machine writes each element of $S$ exactly once. 

\subsection{Group-theoretical decision problems} 
 We refer to \cite{miller_iii_decision_1992} for an authoritative survey of decision problems in group theory.
These come in two flavors. The Word Problem, the Conjugacy Problem and the Subgroup Membership Problem are all \emph{local}, in the sense that they concern the relations of collections of elements of a fixed given finitely presented group.  In contrast, the Isomorphism Problem is an example of  a \emph{global} decision problem, in the sense that it concerns relations between groups themselves.

\subsubsection{Local decision problems}
We start with the local problems that we will consider.
Fix a countable infinite alphabet $\mathcal{A}$ and a subset $A\subseteq\mathcal{A}$; below, $A$ usually is finite.
All  presentations of groups considered below are assumed to have the form $\ll A\mid R\rr$ where $A\subseteq\mathcal{A}$. By
a {\it word}\/ on $A$ we mean an element of 
the free group~$F(A)$ generated by~$A$; that is,
a formal expression of the form $a_1^{\epsilon_1} a_2^{\epsilon_2}\cdots a_n^{\epsilon_n}$
where $a_1,\dots,a_n\in A$ and $\epsilon_1,\dots,\epsilon_n\in\{\pm 1\}$.
If $A$ is a generating set for a group $\pi$, then there is a surjective group morphism $F(A)\to \pi$ with $a\mapsto a$ for each $a\in A$;
given an element $g$ of $\pi$, every word $w\in F(A)$ which maps to $g$ under this morphism is said to {\it represent}\/ the element $g$ of $\pi$.

\medskip
\noindent
We now give  more  precise formulations of Dehn's decidability questions stated in the beginning of the introduction.
Throughout, we assume that a fixed encoding of finite presentations $\ll A\mid R\rr$ of groups and of finite collections of words from $F(\mathcal{A})$ as inputs for Turing machines has been chosen once and for all.

\begin{definition}
Let $\pi$ be a group  with finite generating set $A$.
\begin{enumerate}
\item  The \emph{Word Problem} in $\pi$ asks for an algorithm
  that takes as input a  word $w$ on $A$ and determines whether  $w$ represents the identity element of $\pi$. Note that whether the Word Problem for $\pi$ is solvable does not depend on the choice of the (finite) generating set~$A$ for $\pi$.
     \item    The \emph{Conjugacy Problem} in $\pi$ asks for an algorithm that takes as input two  words $w_1,w_2\in F(A)$ and determines whether   $w_1$ and $w_2$ represent conjugate elements of $\pi$.  Again, this notion is invariant under changing the (finite) generating set $A$.
\item   The \emph{Membership Problem} in $\pi$ asks for an algorithm that takes as input words $v,w_1,\ldots, w_k$ on $A$ and determines whether  $v$ represents an element of the subgroup $\langle g_1,\ldots,g_k\rangle$ of $\pi$ generated by the elements $g_1,\dots,g_k$ represented by $w_1,\dots,w_k$, respectively.
\end{enumerate}
\end{definition}

\noindent
These `local' decision problems also have \emph{uniform} analogues:

\begin{definition}
Let $\mathcal{G}$ be a class of finitely presentable groups. 
\begin{enumerate}
\item  The \emph{Word Problem} in $\mathcal{G}$ asks for an algorithm
  that takes as input a pair  consisting of a presentation $\ll A\mid R\rr$ of a group $\pi$ in $\mathcal G$ and a  word $w\in F(A)$ and determines whether   $w$ represents the identity element of $\pi$.  
     \item    The \emph{Conjugacy Problem} in $\mathcal{G}$ asks for an algorithm that takes as input 
     a presentation $\ll A\mid R\rr$ of a group $\pi$ in $\mathcal G$ and
     two  words $w_1,w_2\in F(A)$ and determines whether   $w_1$ and $w_2$ represent conjugate elements of $\pi$. 
\item   The \emph{Membership Problem} in $\mathcal G$ asks for an algorithm that takes as input 
a presentation $\ll A\mid R\rr$ of a group $\pi$ in $\mathcal G$ and words $v,w_1,\ldots, w_k\in F(A)$, and determines whether  $v$ represents an element of the subgroup $\langle g_1,\ldots,g_k\rangle$ of $\pi$ generated by the elements $g_1,\dots,g_k$ of $\pi$ represented by $w_1,\dots,w_k$, respectively.
\end{enumerate}
We say that the Word Problem is {\it uniformly solvable}\/ in $\mathcal G$
if there is an algorithm which solves the Word Problem in $\mathcal G$; similarly for the other two decision problems.
\end{definition}

\noindent
Note that in the definition above, we do not assume about $\mathcal{G}$ that there is an algorithm to decide whether a given finite presentation describes a group in $\mathcal G$.
Although the class of $3$-manifold groups is not algorithmically recognizable (see Lemma~\ref{lem:cannot detect 3-mfd gps} below), for $3$-manifold groups, all the algorithms we consider in this paper turn out to be uniform, in the sense that there is a single algorithm that solves the problem in question for any $3$-manifold group, with its presentation  provided as additional input for the algorithm.

\subsubsection{Global decision problems}
The main global group-theoretic decision problem that we   consider is the
Isomorphism Problem:
Consider a class $\mathcal{G}$ of finitely presentable groups, closed under isomorphism.  The \emph{Isomorphism Problem} in $\mathcal G$ asks for an algorithm that takes as input finite subsets $A_1,A_2\subseteq\mathcal{A}$, $R_1\subseteq F(A_1)$ and $R_2\subseteq F(A_2)$ such that the groups $\pi_i:=\langle A_i\mid R_i\rangle$ are in $\mathcal{G}$, and determines whether   $\pi_1$ and $\pi_2$ are isomorphic.
The next lemma shows that if we have determined  that $\pi_1$ and $\pi_2$ are indeed isomorphic, then an isomorphism between them can be found algorithmically. Note that a group morphism $F(A)\to F(A')$ is uniquely specified by the images of the generators $a\in A$ of $F(A)$, and thus can be encoded as the output of a Turing machine.

\begin{lemma}\label{lem:findiso}
There exists an algorithm which, given finite presentations $\ll A\mid R\rr$ and $\ll A'\mid R'\rr$ of
isomorphic groups $\pi$, $\pi'$ as input, constructs a group morphism $F(A)\to F(A')$ which induces an isomorphism $\pi\to\pi'$.
\end{lemma}

\begin{proof}
We denote by $N$ the subgroup of $F(A)$ normally generated by $R$; thus $\pi\cong F(A)/N$.
Similarly we define $N'$.
Note that group morphisms $\varphi\co F(A) \to F(A')$ and $\varphi'\co F(A')\to F(A)$ gives rise to 
group morphisms $\pi\to \pi'$ and $\pi'\to\pi$ which are mutually inverse if and only if the following two conditions hold:
\bn
\item  $\varphi(g)\in N'$ for all $g\in R$  and $\varphi'(g')\in N$  for all $g'\in R'$, and
\item $\varphi'(\varphi(a))a^{-1}\in N$ for all $a\in A$ and  $\varphi(\varphi'(a'))(a')^{-1}\in N'$ for all $a'\in A'$.
\en
We now simultaneously enumerate all words in $N$ and $N'$ as well as all group morphisms $\varphi\colon F(A)\to F(A')$ and $\varphi'\colon F(A')\to F(A)$, and check whether conditions (1) and (2) are satisfied. By assumption there exists an isomorphism $\pi\to \pi'$, and
so after finitely many steps we will find a pair
 $\varphi$,~$\varphi'$ as above, satisfying (1) and (2).
 \end{proof}

\noindent
A similar argument as in the previous proof shows that we may let $\ll A\mid R\rr$  and  $\ll A'\mid R'\rr$  range over 
a recursively enumerable set of finite presentations:

\begin{lemma}\label{lem:findiso, rec enum}
Let $\mathcal{R}$, $\mathcal{R'}$  be recursively enumerable sets of finite presentations for groups, and suppose some group represented by an element of $\mathcal R$ is isomorphic to some group represented by an element of $\mathcal R'$. Then there
exists an algorithm which finds presentations $\ll A\mid R\rr$ in~$\mathcal R$ and 
$\ll A'\mid R'\rr$ in~$\mathcal R'$  for groups $\pi$ and $\pi'$, respectively, and also an isomorphism $\pi\to\pi'$.
\end{lemma}
\begin{proof}
We simultaneously enumerate all presentations $\ll A\mid R\rr$  in $\mathcal R$ and $\ll A'\mid R'\rr$ in $\mathcal{R}'$,
all elements of the normal closure $N$ of $R$ in $F(A)$ and all elements of the normal closure $N'$ of $R'$ in $F(A')$, 
as well as all group morphisms $\varphi\colon F(A)\to F(A')$ and $\varphi'\colon F(A')\to F(A)$, and check whether conditions (1) and (2) in the proof of Lemma~\ref{lem:findiso} are satisfied.  
\end{proof}

\noindent
The following lemma is used in Section~\ref{sec:membership problem}.
 
\begin{lemma}\label{lem:finiteindexkernelcokernel}
There exists an algorithm which upon input of finite
presentations of groups $\Gamma$ and $\pi$, finitely many elements of $\pi$ generating
a finite index subgroup $\pi_0$ of $\pi$, and a group morphism $\varphi\colon\Gamma\to\pi$, produces a finite presentation for $\Gamma_0:=\varphi^{-1}(\pi_0)$ and a set of coset representatives for $\Gamma_0$ in~$\Gamma$.  \end{lemma}

\begin{proof}
By going through all epimorphisms from $\pi$ to finite groups we first find a surjective group morphism $\psi\colon \pi\to G$ to a finite group $G$ and a subgroup $G_0$ of $G$ such that $\pi_0=\psi^{-1}(G_0)$. Replacing $\pi$, $\pi_0$, $\varphi$ by $G$, $G_0$, $\psi\circ\varphi$, respectively, we can thus assume that $\pi$ is finite.
We can compute~$\varphi(\Gamma)$ and $\varphi(\Gamma_0)=\pi_0\cap\varphi(\Gamma)$, and find coset representatives for $\varphi(\Gamma_0)\leq \varphi(\Gamma)$. 
We then compute preimages of these coset representatives under the map~$\varphi$, which are then coset representatives for $\Gamma_0\leq\Gamma$.
Using these coset representatives, the Reidemeister--Schreier process~\cite[Chapter~II, Proposition~4.1]{lyndon-schupp} allows us to find a finite presentation for $\Gamma_0$.  
\end{proof}

\noindent
The  Homeomorphism Problem is a topological analogue of the Isomorphism Problem. For this, consider a class $\mathcal{M}$ of compact, triangulable manifolds.  The \emph{Homeomorphism Problem} in $\mathcal{M}$ asks for an algorithm that takes as input triangulations for two manifolds $M_1,M_2\in\mathcal{M}$ and determines whether or not $M_1$ and $M_2$ are homeomorphic.

\section{Basic Definitions of Group Theory and $3$-Manifold Topology}
\label{section:basicdefinitions}

\noindent
As mentioned in the introduction,
the intended audience for this paper consists of both group theorists and $3$-manifold topologists. In order not to clutter the paper with definitions which are obvious to many, but perhaps not all,  readers,  in this section we summarize some of the basic definitions in $3$-manifold topology and group theory which  come up later in the paper. 

\subsection{Special classes of $3$-manifolds}
We start out with the definition of special classes of $3$-manifolds, which form the ``building blocks'' of arbitrary $3$-manifolds (see Section~\ref{section:3manifolds} below).

First of all, a \emph{spherical  $3$-manifold}  is  the quotient of $S^3$ by a finite subgroup~$\Gamma$ of $SO(4)$ acting freely by rotations on $S^3$; the fundamental group of such a spherical $3$-manifold is   isomorphic to $\Gamma$.
Examples of spherical $3$-manifolds are given by the Poincar\'e homology sphere and by lens spaces. Recall that given coprime integers $p,q\geq 1$
the corresponding \emph{lens space $L(p,q)$ } is 
defined as 
\[ L(p,q):=S^3\,\big/\,(\Z/p\Z)= \big\{ (x,y)\in \C^2:|x|^2+|y|^2=1 \big\}\,\big/\,(\Z/p\Z),\]
where $k+p\Z\in \Z/p\Z$ ($k\in\Z$) acts on $S^3$ by 
$$(x,y)\mapsto (xe^{2\pi ik/p},ye^{2\pi ikq/p}).$$
The fundamental group of $L(p,q)$ is isomorphic to $\Z/p\Z$.

A \emph{Seifert fibered manifold} 
is a $3$-manifold $N$ together with a decomposition into disjoint simple closed curves (called \emph{Seifert fibers}) such that each Seifert fiber has a tubular neighborhood that forms a standard fibered torus.
The \emph{standard fibered torus} corresponding to a pair of coprime integers $(a,b)$ with $a\geq 1$ is the surface bundle of the automorphism of a disk given by rotation by an angle of $2\pi b/a$, equipped with the natural fibering by circles.
Every  spherical $3$-manifold  is a Seifert fibered space (see, e.g., \cite{Sc83b}). Further examples 
of Seifert fibered spaces are given by \emph{$\Nil$-manifolds,} which are by definition  $3$-manifolds that are finitely covered by a torus bundle over~$S^1$ whose monodromy action on the first homology of the torus is represented by an upper triangular matrix which is not diagonal.

A Seifert fibered space is called \emph{small} if the base orbifold is a sphere with at most three cone points.

Later on we will also consider \emph{$\Sol$-manifolds}; these are $3$-manifolds  which are finitely covered by a torus bundle over $S^1$ such that the monodromy has real eigenvalues $\lambda$,~$\lambda^{-1}$ with $\lambda>1$.
$\Sol$-manifolds are not Seifert fibered.

The other building blocks for $3$-manifolds are \emph{hyperbolic $3$-manifolds}; 
these are the $3$-manifolds whose interior admits a complete metric of constant negative curvature $-1$.

Thurston showed that, up to a certain equivalence,  there exist precisely eight $3$-dimensional geometries that model compact $3$-manifolds. These geometries are: the  $3$-sphere with the standard spherical metric, Euclidean $3$-space, 
hyperbolic $3$-space, $S^2\times \R$,  $\Bbb{H}^2\times \R$,  the universal cover $\widetilde{\sl(2,\R)}$ of $\sl(2,\R)$, 
 and two further geometries called $\Nil$ and $\Sol$.  We refer to \cite{Sc83a} for details. 
 A $3$-manifold is called \emph{geometric} if it is an $X$-manifold for some geometry~$X$. By \cite{Sc83a} a $3$-manifold is geometric if and only if it is either a $\Sol$-manifold, or hyperbolic or Seifert fibered.

\subsection{Combining $3$-manifolds}
Given oriented $3$-manifolds $M$ and $N$ we denote the connected sum of $M$ and $N$ by $M\# N$. Let
$N$ be a $3$-manifold.
We say that  $N$ is \emph{prime}  if $N$ cannot be written as a non-trivial connected sum of two manifolds,
i.e., if $N \cong N_1\# N_2$, then $N_1 \cong S^3$ or $N_2 \cong S^3$. (Here,~``$\cong$'' denotes ``homeomorphic,'' and below we often identify homeomorphic $3$-manifolds.)
Moreover,~$N$ is called \emph{irreducible}   if every embedded $S^2$ bounds a $3$-ball. Every irreducible $3$-manifold is prime.
Also, if $N$ is an orientable prime $3$-manifold with no spherical boundary components, then by \cite[Lemma~3.13]{Hel76} either~$N$ is irreducible or $N=S^1\times S^2$. 

Recall that in this paper,
all surfaces in $3$-manifolds are assumed to be compact and
properly embedded. A connected surface $\Si$ in a $3$-manifold $N$ is called \emph{incompressible} if the inclusion induced group morphism $\pi_1\Si\to \pi_1N$ is injective. 
A  \emph{Haken manifold}  is an  orientable and irreducible $3$-manifold which 
 contains an  incompressible, orientable surface $\Sigma$ not homeomorphic to $D^2$ or $S^2$.
 
Finally, given a property $\PP$ of manifolds we say that a $3$-manifold  is \emph{virtually $\PP$} if it admits a (not necessarily regular) finite cover which has the property~$\PP$.
\medskip

\subsection{Definitions from group theory} Let $\PP$ be a property of groups and~$\pi$ be a group.
As for manifolds,  we say that $\pi$ is \emph{virtually $\PP$} if $\pi$ admits a (not necessarily normal) subgroup of finite index that satisfies~$\PP$.
Furthermore, we say that $\pi$ is \emph{residually $\PP$} if given any $g\in\pi$ with $g\neq 1$ there exists a  morphism $\a\colon \pi\to G$ onto a group $G$ that satisfies $\PP$  such that $\a(g)\neq 1$.   A case of particular importance is when $\PP$ is the class of finite groups, in which case $\pi$ is said to be \emph{residually finite}.

We say that a subset $S$ of $\pi$ is \emph{separable} if  for any $g\in \pi\setminus S$, there exists a morphism $\a\colon\pi\to G$ to a finite group with $\a(g)\not\in \a(S)$. We say that $\pi$ is \emph{locally extended residually finite \textup{(}LERF\textup{)}} (or \emph{subgroup separable}) if any finitely generated subgroup of $\pi$ is separable in this sense.  Likewise, we say that $\pi$ is \emph{conjugacy separable} if every conjugacy class in $\pi$ is separable.
 
Finally, let $\Gamma$ be a subgroup of $\pi$. We say that  $\Gamma$  is a \emph{retract} (of $\pi$)
if there exists a \emph{retraction} of $\pi$ onto $\Gamma$,  that is, a group morphism $\pi\to \Gamma$ which is the identity on~$\Gamma$. We say that $\Gamma$ is a \emph{virtual retract} if there exists a finite index subgroup of $\pi$ which contains $\Gamma$ as a retract.

\section{A Quick Trip through $3$-Manifold Topology} 
\label{section:3manifolds}

\noindent
In this section we   recall some of the key results in $3$-manifold topology.
Along the way we    draw some first conclusions for the algorithmic study of $3$-manifolds.

\subsection{Moise's theorem}\label{section:moise}

The first theorem of this section is a consequence of~\cite{Mo52,Mo77}. We refer to \cite{AFW12} for precise references.
The theorem  says in particular that in the classification of $3$-manifolds it does not matter whether we work in the category
of topological, triangulable or smooth manifolds.

\begin{theorem}[Moise]\label{thm: Moise} \label{thm:moise}
Let $N$ be a  topological $3$-manifold. Then 
\bnc
\item $N$ admits a finite triangulation, i.e., $N$ is homeomorphic to a finite simplicial complex; any two triangulations are related by a finite sequence of subdivisions and isotopies; and
\item $N$ admits a smooth structure; any two smooth structures give rise to diffeomorphic manifolds.
\enc
\end{theorem}

\noindent
We list various consequences of this theorem:

\begin{corollary}\label{cor: $3$-manifolds are re}
The set of finite simplicial complexes which are homeomorphic to  closed $3$-manifolds is recursively enumerable. Furthermore, the set of finite simplicial complexes which are homeomorphic to  closed orientable $3$-manifolds is recursively enumerable.
\end{corollary}

\begin{proof}
The first statement  follows immediately from Moise's Theorem and the fact that there is an algorithm to determine whether a given finite simplicial complex 
represents a $3$-manifold.  The algorithm in question simply checks that the link of each vertex is a $2$-sphere.  The details are left to the reader.  (See the proof of \cite[Lemma~5.4]{GMW12}.)
The second statement is proved in exactly the same way, except that now we also compute the third homology group to check for orientability.
\end{proof}

\noindent
It is clear that, given a triangulation of a  $3$-manifold, we may write down a finite presentation for its fundamental group, coming from the $2$-skeleton. Hence the previous corollary implies that there is a recursively enumerable set of finite presentations of $3$-manifold groups which contains a presentation of each $3$-manifold group. 
The next result shows that we can also pass from groups to (closed) $3$-manifolds.

\begin{lemma}\label{lem: Groups to closed $3$-manifolds}
There is an algorithm that takes as input a finite presentation for a group $\pi$ and outputs a closed $3$-manifold
 $N$ with $\pi_1N\cong\pi$, if one exists.
\end{lemma}

\begin{proof}
By Corollary~\ref{cor: $3$-manifolds are re} and the remark following it, the set of triangulations of closed $3$-manifolds  is
recursively enumerable, and from a triangulation for a closed $3$-manifold $N$, a finite presentation for $\pi_1N$ can be computed. Hence the claim follows from the argument used in the proof of Lemma~\ref{lem:findiso, rec enum}.
 (The algorithm will not terminate if it is fed a finite presentation of a group that is not a fundamental group of a closed $3$-manifold.)
\end{proof}

\noindent
In contrast to the previous lemma, we have:

\begin{lemma}\label{lem:cannot detect 3-mfd gps}
There is \emph{no} algorithm that takes as input a finite presentation for a group $\pi$ and decides whether $\pi$ is the fundamental group of a closed $3$-manifold.
\end{lemma}

\noindent
  This is an immediate consequence of the fact that although there is no algorithm to decide whether a given finite presentation describes the trivial group \cite{adyan57a, adyan57b, rabin}, there is such an algorithm for the class of $3$-manifold groups (see Corollary~\ref{cor:decide triviality} below).
One can also show Lemma~\ref{lem:cannot detect 3-mfd gps} by observing that being a $3$-manifold group is a `Markov property' in the sense of \cite{rabin},
using the existence of finitely presentable groups which cannot appear as subgroups of $3$-manifold groups (for example, the well-known Baumslag--Solitar groups~\cite{BS62}, or $\mathbb{Z}^4$).


\medskip
\noindent
 In fact, the Word Problem is the only obstacle in the way of algorithmically recognizing closed geometric $3$-manifolds:  

\begin{theorem}[Groves--Manning--Wilton, \cite{GMW12}]
Let $\mathcal G$ be any class of finitely presentable groups 
with uniformly solvable Word Problem.
Then there is an algorithm which takes as input a finite presentation of a group from $\mathcal G$ and decides
whether the group presented by it is the fundamental group of a closed geometric $3$-manifold.
\end{theorem}

\noindent
In a forthcoming paper, the corresponding theorem is proved in the remaining, non-geometric, case \cite{}.  So for example, one can determine algorithmically whether a residually finite finitely presented group is a $3$-manifold group (cf.~Lemma~\ref{lem:resfinite} below).
 
\medskip
\noindent
Finally we observe that given the fundamental group $\pi$ of a  $3$-manifold, we may find a {\it closed}\/ $3$-manifold whose fundamental group retracts onto $\pi$.

\begin{lemma}
There is an algorithm that takes as input a finite presentation for the fundamental group $\pi$ of a  $3$-manifold and outputs a triangulation for a closed $3$-manifold $N$ together with an inclusion $\pi\hookrightarrow\pi_1N$ and a retraction $\pi_1N\to\pi$.
\end{lemma}

\begin{proof}
Let $M$ be a  $3$-manifold, and let $N$ be the double of $M$. Then $N$ is closed;
note that $N$ is orientable if and only if $M$ is orientable.
The `folding map' $N\to M$ is a retraction onto $M$, hence induces a left inverse $\pi_1N\to \pi_1M$ to the morphism $\pi_1M\to \pi_1N$ induced by inclusion.

The algorithm enumerates all closed $3$-manifolds $N$ and all group morphisms $i\co \pi\to\pi_1N$ and  $\rho\co \pi_1N\to\pi$ satisfying $\rho\circ i=\operatorname{id}_\pi$. By the above discussion such  $N$, $i$ and $\rho$ exist. An algorithm as in the proof of Lemma~\ref{lem:findiso}  will eventually find them.
\end{proof}

\noindent
A similar argument as in the proof of the previous lemma, letting $N$ run through all closed orientable $3$-manifolds instead of all closed $3$-manifolds,
shows:

\begin{lemma}\label{lem: Groups to retracts}
There is an algorithm that takes as input a finite presentation for the fundamental group $\pi$ of an orientable  $3$-manifold and outputs a triangulation for a closed orientable $3$-manifold $N$ together with an inclusion $\pi\hookrightarrow\pi_1N$ and a retraction $\pi_1N\to\pi$.
\end{lemma}

\subsection{The Prime Decomposition Theorem}

The following theorem allows one to reduce many questions about $3$-manifolds to the case of prime $3$-manifolds.
The existence of a prime decomposition is due to Kneser~\cite{Kn29}, and its uniqueness is due to Milnor~\cite[Theorem~1]{Mil62}.

\begin{theorem}[Prime Decomposition Theorem] \label{thm:prime} 
Let $N$ be an oriented $3$-manifold with no spherical boundary components. Then 
\bnc
\item there exists a decomposition $N\cong N_1\# \cdots \# N_r$ where the $3$-mani\-folds $N_1,\dots,N_r$ are oriented prime $3$-manifolds;
\item if $N\cong N_1\# \cdots \# N_r$ and $N\cong N_1'\# \cdots \# N_s'$ where the $3$-manifolds $N_i$ and $N_i'$  are oriented and prime, then $r=s$ and \textup{(}after reordering\textup{)} there exist orientation-preserving diffeo\-mor\-phisms $N_i\to N_i'$.
\enc
\end{theorem}

\noindent
The following fact says that at least for closed $3$-manifolds, the prime decomposition can be computed;
this is due 
to Jaco--Rubinstein \cite{jaco_efficient_2003} and Jaco--Tollefson \cite[Algorithm 7.1]{jaco_algorithms_1995}.

\begin{theorem}\label{thm:findprimedecomposition}
There is an algorithm that takes as input a finite triangulation for a closed, orientable $3$-manifold $N$ and outputs a finite list of triangulations for closed oriented $3$-manifolds $N_1,\ldots,N_r$ with the property that $N_1\#\cdots\# N_r$ is the prime decomposition of $N$.
\end{theorem}

\begin{remark}
The problem of determining whether an orientable $3$-manifold is irreducible (or prime) is decidable in space polynomial in the size of the triangulation. 
This is implicit in the work of Jaco--Rubinstein \cite{jaco_efficient_2003} and it is also proved explicitly by Ivanov~\cite[Theorem~1]{Iv08}.
\end{remark}

\noindent
The Prime Decomposition Theorem implies that the fundamental group of any orientable  $3$-manifold with no spherical boundary component can be written as the free product of fundamental groups of prime $3$-manifolds.
The following theorem can be viewed as a converse.

\begin{theorem}[Kneser Conjecture]\label{thm:kneserconj}
Let $N$ be a  compact, orientable $3$-mani\-fold with incompressible boundary and $\Gamma_1,\Gamma_2\leq\pi_1 N$ with $\pi_1N\cong \Gamma_1*\Gamma_2$. Then there exist compact, orientable $3$-manifolds $N_1$ and $N_2$
with $\pi_1N_i\cong \Gamma_i$ for $i=1,2$ and $N\cong N_1\# N_2$.
\end{theorem}

\noindent
The Kneser Conjecture was first proved by Stallings \cite{St59} in the closed case, and by
Heil \cite[p.~244]{Hei72} in the bounded case.

\subsection{The Geometrization Theorem}

The central  result in $3$-manifold topology is the Geometrization Theorem, which had been conjectured by
Thurston \cite{Th82} and proved by Perelman.

\begin{theorem}[Geometrization Theorem]\label{thm:geom}   
Let $N$ be an orientable, irreducible $3$-manifold with empty or toroidal boundary.  Then there exist disjointly embedded  incompressible tori $T_1,\dots,T_k$ such that each component of~$N$ cut along $T_1\cup \dots \cup T_k$ is  hyperbolic  or Seifert fibered. Any such collection of tori with a minimal number of components is unique up to isotopy.
\end{theorem}

\noindent
Jaco--Shalen \cite{JS79} and Johannson \cite{Jo79} independently showed that $N$ splits along tori into Seifert fibered pieces and `atoroidal' pieces; they furthermore showed uniqueness for a minimal collection of such tori.
These ato\-roi\-dal pieces were then shown by  Perelman \cite{Pe02,Pe03a,Pe03b} to be either small Seifert spaces  or hyperbolic. 
The full details of Perelman's proof can be found in \cite{MT07,MT14}.

The decomposition of $N$ along a minimal collection of tori as in the Geometrization Theorem is called the JSJ (Jaco--Shalen-Johannson) decomposition of $N$. The tori in question are called the \emph{JSJ tori} of $N$ and the components obtained by cutting along the JSJ tori are referred to as the \emph{JSJ components} of $N$.
If all of the JSJ components are Seifert fibered, then one calls $N$  a \emph{graph manifold}.

\begin{remark}
The Geometrization Conjecture has also been formulated for non-orien\-table $3$-manifolds; we refer to \cite[Conjecture~4.1]{Bon02} for details.
To the best of our knowledge this has not been fully proved yet.
\end{remark}

\noindent
Jaco--Tollefson \cite{jaco_algorithms_1995} and also Jaco--Letscher--Rubinstein \cite{JLR02} showed how to  compute JSJ decompositions:

\begin{theorem}\label{thm:findjsjdecomposition}
There is an algorithm that takes as input a finite triangulation for a closed, orientable, irreducible $3$-manifold $N$ and outputs the JSJ decomposition of $N$. 
\end{theorem}

\noindent
A proof of the following theorem is provided by  \cite[Algorithm~8.1]{jaco_algorithms_1995}.

\begin{theorem}\label{thm:isitsfs}
There is an algorithm that takes as input a finite triangulation for a closed, orientable, irreducible $3$-manifold $N$ and determines whether $N$ is Seifert fibered or not.
\end{theorem}

\subsection{Consequences of the Geometrization Theorem}

The subsequent theorem is a consequence of the Geometrization Theorem together with work of Leeb \cite{Le95}; see also \cite[Theorem~2.3]{Be13} for the extension to the case of non-toroidal boundary.

\begin{theorem}\label{thm:npc}
If $N$ is an aspherical, orientable $3$-manifold which is not a closed graph manifold,
then the interior of $N$ 
 admits a complete, non-positively curved, Riemannian metric.
\end{theorem}

\noindent
It is important to note  that there  are graph manifolds that are not non-positively curved.
For example Sol- and Nil-manifolds are not non-positively curved.
We refer to  \cite{BK96a,BK96b} and \cite{Le95} for more information.

\medskip
\noindent
The following theorem  is a consequence of the Geometrization Theorem in combination with the Mostow--Prasad Rigidity Theorem, work of Waldhausen \cite[Corollary~6.5]{Wa68a}, and
Scott \cite[Theorem~3.1]{Sc83b} and classical work on spherical $3$-manifolds.

\begin{theorem}\label{thm:primepi}
Let $N$ and $N'$ be two   orientable, closed, prime $3$-manifolds and let  $\varphi\colon \pi_1N\to \pi_1N'$ be an isomorphism. Then the following hold:
\bnc
\item If $N$ and $N'$ are not lens spaces, then  $N$ and $N'$ are homeomorphic.
\item If  $N$ and $N'$ are not spherical, then there exists a homeomorphism which induces $\varphi$.
\enc
\end{theorem}

\subsection{The Virtually Compact Special Theorem}

We now turn to the theorem which is arguably the most important result in $3$-manifold topology since the proof of the Geometrization Conjecture.
The statement was  proved by Wise \cite{Wi09,Wi12a,Wi12b} for hyperbolic $3$-manifolds with boundary
and for closed hyperbolic $3$-manifolds which admit a geometrically finite surface. The general case of closed hyperbolic $3$-manifolds is 
due to Agol~\cite{Ag13}.

\begin{theorem}[Virtually Compact Special Theorem] \label{thm:akmw}
The fundamen\-tal group of a hyperbolic $3$-manifold is the fundamental group of a virtually-special compact cube complex.
\end{theorem}

\noindent
The definition of a  `special compact cube complex' goes back to Haglund and Wise~\cite{HW08}. We will not repeat the definition here, but we will give a precise algebraic statement which is a consequence of the preceding theorem, and sufficient for our purposes.

In order to do so, we need one more definition.
Given a   graph $G$ with vertex set $V=\{v_1,\dots,v_k\}$ and edge set $E$,
the corresponding \emph{right-angled Artin group} is defined as
\[
\Gamma=\Gamma(G)=\big\langle v_1,\dots,v_k \,\big|\, \text{$[v_i,v_j]=1$ if $(v_i,v_j)\in E$}\bigg. \big\rangle.\]
For example, free groups of finite rank and finitely generated free abelian groups are
right-angled Artin groups. See \cite{charney} for a survey on this class of groups.
The following theorem now follows from Theorem \ref{thm:akmw}.

\begin{theorem}\label{thm:raag}
Let $N$ be a hyperbolic $3$-manifold. Then $\pi_1N$ is  virtually a virtual retract of a right-angled Artin group. This means that there exists a finite index subgroup $\pi'$ of $\pi_1N$, a finite-index subgroup $\Gamma$ of a right-angled Artin group, and an embedding $\pi'\to \Gamma$ which has a left-inverse $\Gamma\to \pi'$.
\end{theorem}

\noindent
We refer to \cite{AFW12} for details.
The Virtually Compact Special Theorem has many striking consequences; for example, it implies by Agol's theorem \cite{Ag08} (see also \cite{FK14}) that any
hyperbolic $3$-manifold with empty or toroidal boundary  is virtually fibered.
Together with 
the Tameness Theorem of Agol \cite{Ag04} and Calegari--Gabai \cite{CG06} 
and work of Haglund \cite{Hag08} one also obtains the following theorem.

\begin{theorem}\label{thm:hypsubgroup}
Let $N$ be a hyperbolic $3$-manifold and $\Gamma$ be a finitely generated subgroup of $\pi:=\pi_1N$. Then one of the following holds:
\bnc
\item either $\Gamma$ is a virtual retract of $\pi$, or
\item there exists a finite index subgroup $\pi'$ of $\pi$ which contains $\Gamma$ as a normal subgroup with $\pi'/\Gamma\cong \Z$.
\enc
In either case, $\Gamma$ is a separable subgroup of $\pi$. 
\end{theorem}

\noindent
In the theorem the former case corresponds to $\Gamma$ being `geometrically finite', whereas the latter corresponds to $\Gamma$ being `geometrically infinite'. 
Here we again refer to our survey \cite{AFW12} for  details and precise references.
Theorem~\ref{thm:hypsubgroup} implies in particular  that 
the fundamental group of a hyperbolic $3$-manifold is subgroup separable.

\section{Decision Problems}\label{section:solutions}

\noindent
This section is organized as follows.  For each decision problem, we start by stating the theorem that describes its solvability in maximum generality and by sketching a proof or giving all the relevant references.  We then go on to give a brief sample of the literature on the problem and  discuss variations.

\subsection{The Word Problem}
The first classical decision problem which was solved for $3$-manifold groups was the Word Problem.

\begin{theorem}\label{thm: Word Problem}
The Word Problem is uniformly solvable in the class of fundamental groups of  $3$-manifolds.
\end{theorem}

\noindent
It is conceptually easiest to prove the theorem by appealing to the following theorem, a consequence 
of the Geometrization Theorem and work of Hempel~\cite{hempel_residual_1987}.

\begin{theorem}\label{thm:Hempel}
The fundamental group of each $3$-manifold is residually finite.
\end{theorem}

\noindent
Theorem~\ref{thm: Word Problem} is now an immediate consequence of the following well-known observation
of Dyson~\cite{dyson} and Mostowski~\cite{mostowski},
which in essence goes back to McKinsey~\cite{mckinsey}; see~\cite{evans}.

\begin{lemma}\label{lem:resfinite}
The Word Problem is uniformly solvable in the class of finitely presented residually finite groups.
\end{lemma}

\begin{proof}
Let $\pi=\ll A\,|\,R\rr$  be a finite presentation  for a residually finite group
and let $w$ be a word in $A$.  Systematically enumerate all words which are products
of conjugates of elements in $R$ and check whether $w$ is one of them, 
and  simultaneously also enumerate all morphisms $\a\colon\pi\to G$ to finite groups~$G$
and test whether the element $g$ of $\pi$ represented by $w$ is in the kernel of $\a$.
If $w$ represents the identity element of $\pi$, then our first procedure  eventually detects this;
on the other hand, if $w$ represents an element $g\neq 1$ of $\pi$, then by residual finiteness there exists a morphism $\a\colon \pi\to G$ to a finite group with $\a(g)\ne 1$, and our second procedure will  detect this after finitely many steps.
\end{proof}

\noindent
Although this approach gives a clean and uniform solution to the Word Problem, it has the disadvantage that it gives a very poor upper bound for its computational complexity.   Another approach, which gives much better estimates of the complexity of the algorithm involved, uses the notion of automaticity.  We refer to \cite{epstein_word_1992} for the definition of an automatic group and for further details.

For our purposes, the key fact is that the elements of an automatic group can efficiently be put into a \emph{canonical form} \cite[Theorem 2.3.10]{epstein_word_1992}; in particular, the word problem is (efficiently) solvable in automatic groups.

The following result is Theorem 12.4.7 of \cite{epstein_word_1992}.

\begin{theorem}\label{thm:n3automatic}
Let $N$ be an orientable $3$-manifold such that no prime factor admits $\Nil$ or $\Sol$ geometry.  Then $\pi_1N$ is automatic.
\end{theorem}

\noindent
By another result of \cite{epstein_word_1992}, there is an algorithm that finds automatic structures on groups, and it follows that this solution to the word problem is uniform. 

One way to quantify the efficiency of this solution to the Word Problem uses the notion of the Dehn function of a group.  Given a finite presentation  $\langle A\,|\, R\rangle$ of a group~$\pi$ and word $w\in F(A)$ that represents the identity element in $\pi$, the \emph{area} of $w$ is defined to be the minimal $n$ such that
\[
w=\prod_{i=1}^n g_ir_i^{\pm 1}g_i^{-1}\qquad\text{where $g_i\in F(A)$ and $r_i\in R$.}
\]
The \emph{Dehn function $\delta_{\langle A\,|\, R\rangle}\colon\N\to\N$} of the presentation $\langle A\,|\, R\rangle$ is   defined   by
\[
\delta_{\langle A\,|\, R\rangle}(n)=\max\big\{ \operatorname{Area}(w) : l_A(w)\leq n,\ w=_{\pi}1 \big\},
\]
where $l_A(w)$ is the length of the word $w$ in $F(A)$; that is, $\delta(n)$ is the maximal area among all words of length at most $n$ in $F(A)$ which represent the identity element of $\pi$.  Although this definition depends on the presentation~$\langle A\,|\, R\rangle$ of $\pi$, it turns out that the growth type of $\delta_{\langle A\,|\, R\rangle}$ only depends on $\pi$, and so we will often write $\delta_\pi$ instead.
The Word Problem in $\pi$ is solvable if and only if its Dehn function is computable (by a Turing machine).
Roughly speaking, the Dehn function of $\pi$ measures the difficulty of solving the Word Problem in $\pi$.

\begin{definition}
The group $\pi$ is said to satisfy a \emph{linear \textup{(}respectively sub-quadratic, quadratic, exponential\textup{)} isoperimetric inequality} if $\delta_\pi$ is bounded above by a linear (respectively sub-quadratic, quadratic, exponential) function.
\end{definition}

\noindent
The terminology refers to a beautiful \emph{Filling Theorem} which, for a compact Riemannian manifold $N$, relates the growth type of $\delta_{\pi_1N}$ to solutions to Plateau's Problem in the universal cover of $N$ (see \cite{bridson_geometry_2002}).

\medskip
\noindent
The connection to automaticity is given by \cite[Theo\-rem~2.3.12]{epstein_word_1992}:

\begin{theorem}\label{thm: Automatic quadratic}
Every automatic group satisfies a quadratic isoperimetric inequality.
\end{theorem}

\noindent
The combination of Theorems \ref{thm:n3automatic} and \ref{thm: Automatic quadratic} now gives  the following corollary.

\begin{corollary}\label{cor: Quadratic isoperimetric inequality}
Let $N$ be an orientable $3$-manifold such that no prime factor admits $\Nil$ or $\Sol$ geometry. Then $\pi_1N$ satisfies a quadratic isoperimetric inequality.
\end{corollary}

\noindent
This result is optimal, as demonstrated by the following consequence of the Geometrization Theorem and a theorem of Gromov \cite[2.3.F]{gromov_hyperbolic_1987} (see also~\cite{papsoglu_sub-quadratic_1995}).

\begin{theorem}
For an  irreducible $3$-manifold $N$, the following are equivalent:
\bnc
\item $\pi_1N$ satisfies a sub-quadratic isoperimetric inequality;
\item $\pi_1N$ satisfies a linear isoperimetric inequality;
\item $N$ is hyperbolic.
\enc
\end{theorem}

\noindent
The $\Nil$ and $\Sol$ cases are also known: $\Nil$-manifolds have cubic Dehn functions  and $\Sol$ manifolds have exponential Dehn functions; see Example~8.1.1 and Theorem~8.1.3 of \cite{epstein_word_1992}, respectively.

 A third approach to the Word Problem is provided by non-positive curvature: if $N$ admits a non-positively curved metric then, by the Filling Theorem alluded to above, $\pi_1N$ admits a quadratic isoperimetric inequality.  By Theorem~\ref{thm:npc} this gives a solution to the Word Problem for a very large class of $3$-manifolds.

\medskip
\noindent
We note some consequences of the solvability of the uniform Word Problem for $3$-manifold groups:

\begin{corollary}\label{cor:decide triviality}
There is an algorithm which upon input of a finite presentation of a $3$-manifold group $\pi$, decides whether $\pi$ is trivial, respectively abelian. 
\end{corollary}
\begin{proof}
A finitely generated group $G$ is trivial iff each generator is trivial, and $G$ is abelian iff each pair of generators commutes. Both conditions can be effectively verified using an algorithm for the Word Problem.
\end{proof}

\subsection{The Conjugacy Problem}
The second of Dehn's decision problems  also has a positive solution for $3$-manifold groups:

\begin{theorem}
The Conjugacy Problem is uniformly solvable in the class of fundamental groups of  $3$-manifolds.
\end{theorem}

\noindent
Pr\'{e}aux, extending Sela's work on knot groups \cite{sela_conjugacy_1993}, proved that the Conjugacy Problem is solvable, first for the fundamental groups of orientable $3$-manifolds \cite{preaux_conjugacy_2006}, and then also for the fundamental groups of non-orientable  $3$-manifolds \cite{preaux_conjugacy_2012}.  (Note that, in contrast to many other group properties, solvability of the Conjugacy Problem does not automatically pass to finite extensions.)  Although he does not explicitly state it, Pr\'{e}aux's solution to the Conjugacy Problem for $3$-manifold groups is uniform.

\medskip 
\noindent
For orientable $3$-manifolds we can also give a proof of this theorem which is analogous to the first proof in the previous section.   More precisely, building on the Virtually Compact Special Theorem and work of Minasyan \cite{Min12} it was shown in  \cite[Theorem~1.3]{hamilton_separability_2011} that  the  fundamental group  of any  orientable $3$-manifold $N$ is \emph{conjugacy separable.}\/  This means that given any non-conjugate $g,h\in \pi_1N$ there exists a morphism $\a\colon \pi_1N\to G$ to a finite group $G$ such that $\a(g)$ and $\a(h)$ are non-conjugate. It now follows from an argument similar to the proof of  Lemma~\ref{lem:resfinite} that the Conjugacy Problem is uniformly solvable for fundamental groups of orientable $3$-manifolds.  

\medskip
\noindent
As in the last section, if a $3$-manifold $N$ admits a non-positively curved metric, then \cite[Theorem III.$\Gamma$.1.12]{bridson_metric_1999}
gives another solution to the conjugacy problem for $\pi_1N$. However, as the constants involved in this theorem depend on the non-positively curved metric on $N$, this does not \emph{a priori} give a \emph{uniform} solution.

\medskip
\noindent
Finally, the Conjugacy Problem is not known to be solvable for automatic groups \cite[Open Question 2.5.8]{epstein_word_1992}.  The Conjugacy Problem is known to be solvable for biautomatic groups, but it is unknown whether automatic $3$-manifold groups are biautomatic \cite[Question 9.33]{AFW12}. 

\subsection{The Membership Problem}\label{sec:membership problem}

The following theorem was recently proved in \cite{FW14}.
It should be contrasted with the fact that the Membership Problem is not solvable even for fairly simple groups like 
the direct product $F_2\times F_2$ of two copies of the free group $F_2$ on two generators
\cite{Mi58,Mi66}.

\begin{theorem}\label{thm: MP}\label{thm:mp}
The Membership Problem is uniformly solvable in the class of fundamental groups of $3$-manifolds.
\end{theorem}

\noindent
In Theorem \ref{thm:hypsubgroup} we saw that fundamental groups of hyperbolic $3$-mani\-folds are subgroup separable.   Using the argument of the proof of Lemma~\ref{lem:resfinite} we then obtain a solution to the uniform Membership Problem for fundamental groups of hyperbolic $3$-manifolds. This approach works also for fundamental groups of Seifert fibered spaces, which are subgroup separable by work of Scott~\cite{Sc78}.  The argument can also be used  for special classes of subgroups, e.g., subgroups carried by embedded surfaces~\cite{PW14}, but this approach does not work in general.  For example, Niblo--Wise \cite[Theorem~4.2]{NW01}  showed that the fundamental groups of most graph manifolds are not subgroup separable.  We thus see that we cannot hope to prove Theorem~\ref{thm: MP} in the general case
by appealing to separability properties only. 

\medskip
\noindent
In the following we will summarize the proof 
of Theorem \ref{thm:mp}.
 We refer to \cite{FW14} for full details and for a careful discussion of the fact that we can give a uniform solution to the membership problem.

\medskip
\noindent
In the proof of 
Theorem~\ref{thm: MP} we employ the following useful lemma.

\begin{lemma}\label{lem: MP passing to fi subgroups}
Let $\mathcal{G}$ be a class of finitely presentable groups, 
and let $\mathcal{R}$ be a set of finite presentations of groups. Suppose that every group in $\mathcal{G}$ is virtually isomorphic to a group presented by some element of~$\mathcal{R}$.  If
\begin{compactitem}
\item the Membership Problem is uniformly solvable in the class of groups presented by elements of $\mathcal{R}$, and 
\item $\mathcal{R}$ is recursively enumerable,
\end{compactitem}
then the Membership Problem is also uniformly solvable in $\mathcal{G}$.
\end{lemma}
\begin{proof}
Consider a group $\pi$ in $\mathcal{G}$.  Using the Reidemeister--Schreier procedure, we may enumerate all subgroups of finite index in $\pi$.  Because $\mathcal{R}$ is recursively enumerable, we will eventually find one, $\pi_0$ say, that we may confirm is in~$\mathcal{R}$, using Lemma~\ref{lem:findiso, rec enum}.

Let $H$ be a finitely generated subgroup of $\pi$, specified by a finite set of elements. By Lemma~\ref{lem:finiteindexkernelcokernel}, we may compute a generating set for $H_0= H\cap\pi_0$ and a set of (left) coset representatives $h_1,\ldots, h_k$ for $H_0$ in $H$.  Now, let $g\in \pi$ be given. For each $i$, we may determine whether  $h^{-1}_ig\in\pi_0$, since membership in subgroups of finite index is always decidable.  If there is no such $i$ then evidently $g\notin H$.  Otherwise, if $h^{-1}_ig\in\pi_0$ then, using the solution to the Membership Problem in $\pi_0$, we may determine whether   $h^{-1}_ig\in H_0$. Since $g\in H$ if and only if $h^{-1}_ig\in H_0$ for some $i$, this solves the Membership Problem in $\pi$.
\end{proof}

\noindent
Solvability of the Membership Problem in fundamental groups of graphs of groups was addressed by Kapovich--Miasnikov--Weidmann \cite{kapovich_foldings_2005}, who gave the following definition. See, e.g., \cite[Section~1.2]{AF13} or
\cite{serre}
for the definition of a graph of groups.

\begin{definition}
Consider a graph of finitely generated groups.  In the following, $G_v$ is a vertex group and $G_e$ is an incident edge group. Such a graph of finitely generated groups is \emph{benign} if:
\begin{itemize}
\item[(B1)] for each vertex $v$ there is an algorithm to test membership of double cosets of the form $HG_e$ in $G_v$, where $e$ is an incident edge and $H$ is a finitely generated subgroup of $G_v$;
\item[(B2)] for each edge $e$ the group $G_e$ is \emph{slender}, meaning that every subgroup of $G_e$ is finitely generated;
\item[(B3)] for each edge $e$ there is an algorithm to solve the Membership Problem in the edge group $G_e$;
\item[(B4)] for each vertex $v$ and incident edge $e$, there is an algorithm to compute generating sets for intersections $H\cap G_e$ where $H$ is a finitely generated subgroup of $G_v$.
\end{itemize}
\end{definition}

\noindent
The following is~\cite[Theorem 5.13]{kapovich_foldings_2005}:

\begin{theorem}[Kapovich--Miasnikov--Weidmann]\label{thm: KMW}
The Membership Problem is solvable for fundamental groups of benign graphs of groups in which every vertex group has solvable Membership Problem.
\end{theorem}

\begin{remark}
The given solution is not \emph{a priori} uniform.  However, it depends only on the graph of groups and the algorithms guaranteed by the hypotheses.  If these can be found uniformly, then the solution to the Membership Problem is indeed uniform.
\end{remark}

\noindent
For future reference we record the following immediate corollary first proved by Miha{\u\i}lova~\cite{Mi59,Mi68}; 
see also  \cite[Corollary 5.16]{kapovich_foldings_2005}.

\begin{corollary}\label{cor:solvablempproblemandfreeproducts}
Solvability of the Membership Problem is preserved under taking free products.
\end{corollary}

\noindent
We  now argue  that the JSJ graph of groups of a closed $3$-manifold is benign.

Condition (B1) follows from an argument as in Lemma~\ref{lem:resfinite}
and the fact that given a  fundamental group of a Seifert fibered space or a hyperbolic $3$-manifold, any product of two finitely generated subgroups is separable. In the Seifert fibered case this was proved by 
Niblo~\cite{Ni92} building on the aforementioned work of  Scott~\cite{Sc78}. 
In the hyperbolic case this is an extension of 
 Theorem~\ref{thm:hypsubgroup}, which follows from work of Wise~\cite[Theorem~16.23]{Wi12a} combined with work of Hruska~\cite[Corollary~1.6]{Hr10}.
 
Condition~(B2) is immediate, and (B3) follows from the fact that the Membership Problem in $\Z^2$ can be solved easily using basic algebra.

Therefore, to prove the solvability of the Membership Problem in $3$-manifold groups, it now remains to address~(B4).  We deal with the Seifert fibered and hyperbolic cases separately, in the following two lemmas.

\begin{lemma}\label{lem: Intersections for hyperbolic manifolds}
Let $N$ be a hyperbolic $3$-manifold with  toroidal boundary. Let~$P$ be a cusp subgroup of $\pi=\pi_1N$.  There is an algorithm, uniform in $\pi$ and $P$, that takes as input a finite set of elements that generate a subgroup~$\Gamma$ of~$\pi$, and computes a generating set for $\Gamma\cap P$.
\end{lemma}

\begin{proof}
By Theorem~\ref{thm:hypsubgroup}, one of the following happens:
\begin{enumerate}
\item either there exists a
 finite-index subgroup $\pi_0$ of $\pi$ and a retraction $\rho\colon \pi_0\to\Gamma$; or
\item  there exists a finite-index subgroup $\pi_0$ and a morphism $p\colon \pi_0\to \Z$ such that $\Gamma=\ker p$.
\end{enumerate}
We now run two algorithms simultaneously:
 a naive search using the Reide\-meister--Schreier algorithm
 to find   $\pi_0$ and $\rho$ as in Case~(1), and
 a naive search using the Reidemeister--Schreier algorithm and the algorithm
 used in the proof of Lemma~\ref{lem:findiso, rec enum} looking for $\pi_0$ and $p$ as in Case~(2). 
 Since Case~(1) or (2) hold, one of these algorithms will terminate.
  By Lemma~\ref{lem:finiteindexkernelcokernel},  in either case we can compute generators for $P_0=\pi_0\cap P$.

In Case~(2), $\Gamma\cap P=(\ker p)\cap P_0$, which can be computed by standard linear algebra.

Suppose that we are in Case~(1). We note that $\Gamma\cap P=\rho(P_0)\cap P_0$. 
Using the solution to the Word Problem in $\pi$ we can determine whether all generators of $\rho(P_0)$ and $P_0$ commute, i.e.,  whether  $[\rho(P_0),P_0]=1$. 

First suppose that  $[\rho(P_0),P_0]=1$. It follows from the well-known fact 
that $P_0$ is maximal abelian in $\pi_0$
(see, e.g., \cite[Theorem~3.1]{AFW12}) that then $\rho(P_0)\subseteq P_0$, which implies that  $\Gamma\cap P=\rho(P_0)$. 

Now suppose that  $[\rho(P_0),P_0]\ne 1$. The fact that $N$ is hyperbolic implies
by \cite[Corollary 3.11]{AFW12} that   the centralizer of any non-identity element in $\pi_0$ is abelian. It now follows that $\rho(P_0)\cap P_0=1$ and so $\Gamma\cap P=1$.
\end{proof}

\noindent
A similar argument deals with the case of a product manifold:

\begin{lemma}\label{lem: Intersections for SF manifolds}
Let $N$ be a compact Seifert-fibered manifold with boundary and let~$P$ be a cusp subgroup of $\pi=\pi_1N$.  There is an algorithm, uniform in $\pi$ and $P$, that takes as input a finite set of elements that generate a subgroup~$\Gamma$ of $\pi$, and computes a generating set for $\Gamma\cap P$. 
\end{lemma}
\begin{proof}
Since $N$ is finitely covered by a product $\Sigma\times S^1$ (where $\Sigma$ is a compact surface with boundary), one quickly reduces to the case $N=\Sigma\times S^1$ (see \cite[Lemma~26]{FW14} for full details).

The case that $\Sigma$ is a disk or an annulus is trivial. We therefore henceforth assume that $\chi(\Sigma)<0$.  Note that, since every finitely generated subgroup of a surface group $\pi_1\Sigma$ is a virtual retract (this is implicit in \cite{Sc78}), it follows easily that every finitely generated subgroup of $\pi_1\Sigma\times \mathbb{Z}$ is a virtual retract.  Therefore, a naive search using the Reidemeister--Schreier algorithm will find a finite-index subgroup $\pi_0$ of $\pi$ and a retraction $\rho\colon\pi_0\to\Gamma$.  
As in the proof of Lemma~\ref{lem: Intersections for hyperbolic manifolds}, we can  compute generators for $P_0=\pi_0\cap P$.

Again, we note that $\Gamma\cap P=\rho(P_0)\cap P_0$.  As before, an explicit computation again determines whether   $[\rho(P_0),P_0]=1$. 
  If so then, just as before, because $P_0$ is maximal abelian we have $\rho(P_0)\subseteq P_0$ and so $\rho(P_0)=\Gamma\cap P$.  If not, then by the commutative transitivity of $\pi_1\Sigma$, we deduce that $\Gamma\cap P=\rho(P_0)\cap P_0$ is contained in the centre $Z_0$ of $\pi_0$ and so it suffices to compute $\rho(P_0)\cap Z_0$. 
  (Here recall that a group is {\it commutative transitive}\/ if $[a,b]=1$ and $[b,c]=1$ imply that $[a,c]=1$.)
   But now $\rho(P_0)\cap Z_0$ can be seen in the abelianization of $\pi_0$, and so can be computed by elementary linear algebra.
\end{proof}

\noindent
We are now ready to prove that the Membership Problem is solvable.

\begin{proof}[Proof of Theorem~\ref{thm: MP}]
We start out by showing that it suffices to show that the Membership Problem is uniformly solvable for the class of  fundamental groups of \emph{orientable} $3$-manifolds.
To see this, 
let $\pi$ be the fundamental group of a $3$-manifold $N$.   Then  $\pi':=\ker(\pi\to H_1(\pi;\F_2))$ is the fundamental group of an orientable $3$-manifold. By Lem\-ma~\ref{lem:finiteindexkernelcokernel}, we can determine a presentation for $\pi'$. Moreover, by Lemma~\ref{lem: MP passing to fi subgroups}, a solution to the Membership Problem for~$\pi'$ also gives a solution to the Membership Problem for $\pi$.  This concludes the proof of the claim.

\medskip
\noindent
By the claim and Lemma~\ref{lem: Groups to retracts}, we may now assume that $\pi$ is the fundamental group of a closed orientable $3$-manifold $N$ and that we furthermore have a triangulation for $N$ available.
By  Theorem~\ref{thm:findprimedecomposition} we can determine the Prime Decomposition of $N$. By  
Corollary~\ref{cor:solvablempproblemandfreeproducts} it  thus suffices to consider the prime components of $N$. 
We can henceforth assume that $N$ is  irreducible.

 Using Theorem \ref{thm:isitsfs} we  determine whether $N$ is Seifert fibered. As   mentioned before,  in this case, $\pi$ is subgroup separable by \cite{Sc78}; this immediately gives a uniform solution to the Membership Problem.  If~$N$ is not Seifert fibered, then we use  Theorem~\ref{thm:findjsjdecomposition} to  determine the JSJ decomposition of $N$.  The corresponding graph of groups decomposition of~$\pi$ is benign by Lemmas~\ref{lem: Intersections for hyperbolic manifolds} and~\ref{lem: Intersections for SF manifolds}. So the result follows from Theorem~\ref{thm: KMW}.
\end{proof}

\begin{remark}
As a further consequence of \cite[Theorem~5.8,~(b)]{kapovich_foldings_2005}, we also get that the \emph{Uniform Finite Presentation Problem} for closed, orientable $3$-manifold groups is solvable. That is, there is an algorithm that takes as input a finite set $S$ of elements of a fundamental group $\pi$ of a closed, orientable $3$-manifold and outputs a finite presentation for the subgroup~$\langle S\rangle$ of $\pi$, and this algorithm is uniform in $\pi$.
\end{remark}

\subsection{The Isomorphism Problem}

Much of the material in this section derives from \cite[Section 10]{sela_isomorphism_1995}. We  refer to \cite{bessieres_geometrisation_2010} for further details.

\begin{theorem}\label{thm: IP}
The Isomorphism Problem for the class of fundamental groups of closed, orientable $3$-manifolds is solvable.
\end{theorem}

\noindent
In the remainder of this section we will sketch the proof of Theorem \ref{thm: IP}.
First, it follows from  Lemma \ref{lem: Groups to closed $3$-manifolds} together with Theorems  \ref{thm:findprimedecomposition} and \ref{thm:kneserconj}  that  it suffices to solve the Isomorphism Problem for fundamental groups of closed, orientable and \emph{irreducible} $3$-manifolds. 

 At this stage, the argument divides into the cases in which $N$ is Haken and $N$ is non-Haken.  First, we need to be able to determine which case we are in.  That we can follows from \cite[Theorem 4.3]{jaco_algorithm_1984}.

\begin{theorem}[Jaco--Oertel]  
There is an algorithm to determine whether  a given closed, irreducible $3$-manifold $N$ is Haken.
\end{theorem}

\noindent
The Haken case is dealt with by the following theorem of Haken \cite{haken_homoomorphieproblem_1962}, Waldhausen \cite{waldhausen_recent_1978}, Hemion \cite{Hen79} and Matveev \cite{Mat03}.

\begin{theorem}[Haken]  
There is an algorithm that determines whether  an input pair of orientable Haken $3$-manifolds are homeomorphic.
\end{theorem}

\noindent
By Theorem \ref{thm:primepi}, this gives rise to a group-theoretic analogue.

\begin{corollary}
The Isomorphism Problem for the class of fundamental groups of closed, orientable, Haken $3$-manifolds is solvable.
\end{corollary}

\noindent
It remains to deal with the non-Haken case which, by Geometrization, divides into three cases: spherical, small Seifert fibered with infinite fundamental group and hyperbolic. 
The next theorem asserts that these can be recognized algorithmically.

\begin{theorem}
There is an algorithm to determine whether a given closed, orientable, non-Haken $3$-manifold is spherical, small Seifert fibered  with infinite fundamental group, or hyperbolic.
\end{theorem}

\noindent
See \cite[Theorem 10.5]{sela_isomorphism_1995} for the proof.  The main ingredient is Papasoglu's algorithm for finding word-hyperbolic structures on a group \cite{papasoglu_algorithm_1996}.  Alternatively, one could use Manning's algorithm \cite{manning_algorithmic_2002}.

\medskip
\noindent
It now remains to solve the Isomorphism Problem in these three cases.  The Isomorphism Problem is easily solvable for finite groups: given finite presentations $\langle A\,|\, R\rangle$ and $\langle A'\,|\,R'\rangle$ for
groups $\pi$, $\pi'$, respectively, first use the Todd--Coxeter algorithm (see, e.g., \cite{sims}) to construct multiplication tables for $\pi$, $\pi'$, using which it is straightforward to determine whether $\pi\cong\pi'$.
The remaining two cases are dealt with by the following theorems of Sela.

\begin{theorem}[Sela]
The Isomorphism Problem is solvable for the infinite groups which are fundamental groups of small Seifert fibered manifolds.
\end{theorem}

\noindent
The proof of this theorem is the content of \cite[pp.~280--281]{sela_isomorphism_1995}.  The problem quickly reduces to the Isomorphism Problem for triangle groups. Using more heavy machinery, one could instead invoke Dahmani--Guirardel's solution to the Isomorphism Problem for word-hyperbolic groups with torsion \cite{dahmani_isomorphism_2011}.

\begin{theorem}[Sela]\label{thm: IP for hyp manifolds}
The Isomorphism Problem is solvable for fundamental groups of closed, orientable hyperbolic $3$-manifolds.
\end{theorem}

\noindent
This is an immediate corollary of the main theorem of \cite{sela_isomorphism_1995}.  Scott and Short \cite{SS12} gave an alternative proof using Manning's algorithm \cite{manning_algorithmic_2002}.

\medskip
\noindent
This completes the proof of Theorem~\ref{thm: IP}.

\subsection{The Homeomorphism Problem}\label{section:homeomorphismproblem}

In several places in the literature, it is stated that an affirmative solution to the Homeomorphism Problem for all closed, orientable $3$-manifolds follows from the Geometrization Theorem and the aforementioned work of Sela \cite{sela_isomorphism_1995}. In fact, we have:

\begin{theorem}\label{thm:homeomorphism}
The Homeomorphism Problem is solvable for orientable, irreducible 
$3$-manifolds  with only incompressible boundary components.
\end{theorem}

\noindent
Note that by Theorem \ref{thm:moise}, two $3$-manifolds are diffeomorphic if and only if they are homeomorphic.
A solution to the Homeomorphism Problem is thus also  a solution to the Diffeomorphism Problem.

\begin{proof}
Let $N$ and $N'$ be  orientable, irreducible $3$-manifolds.
It is well-known how to compute the first Betti numbers $b_1 N$, $b_1 N'$ of $N$ respectively~$N'$. 

First assume that $b_1N\geq 1$ or $b_1N'\geq 1$, say $b_1N\geq 1$.  Evidently $N'$ is homeomorphic to $N$ only if $b_1N'\geq 1$,
so we may also assume $b_1N'\geq 1$.
Since the boundary components  of $N$ and $N'$ are incompressible it follows from a  standard argument (see, e.g., \cite[(C.18)]{AFW12}) that $N$ and $N'$ are Haken. 
A solution to the Homeomorphism Problem for Haken manifolds has been given by Matveev~\cite[Theorem~6.1.1]{Mat03}, building on earlier work of Hemion~\cite{Hen79} and Haken~\cite{haken_homoomorphieproblem_1962}. 

We now suppose that $b_1N=b_1N'=0$. Since $N$ is irreducible it follows that 
$N$ is either the $3$-ball or closed, and similarly for $N'$. Since connected closed orientable surfaces are classified by their Euler characteristics one can easily determine whether the boundaries of $N$ and $N'$ are homeomorphic to~$S^2$. Hence we can now restrict  to the case that $N$ and $N'$ are closed.

By Theorem~\ref{thm: IP} we can determine whether $\pi_1N$ and $\pi_1N'$ are isomorphic.
Theorem~\ref{thm:primepi} implies that if they are, then $N$ and $N'$ are either homeomorphic or they are both lens spaces.
Suppose that $N$ and $N'$ are lens spaces with fundamental group~$\Z/p\Z$. 
By the classification of lens spaces we  can provide a complete, necessarily finite, list of pairwise non-homeomorphic triangulated lens spaces $L_1,\dots,L_k$ with fundamental group $\Z/p\Z$. Since $N$ and $N'$ are PL-isomorphic to precisely one of the $L_i$ we can now determine whether or not $N$ and $N'$ are homeomorphic. Alternatively, one can use Reidemeister's original approach~\cite{Re35} via the torsion invariant to determine whether $N$ and $N'$ are homeomorphic. 
\end{proof}

\noindent
We do not doubt that the hypotheses of Theorem~\ref{thm:homeomorphism} can be relaxed; in particular, most experts agree that a solution to the Homeomorphism Problem among closed, orientable 3-manifolds that are not necessarily irreducible is within the reach of current techniques.  However, the details of such an algorithm do not, as far as we know, appear in the literature.\medskip

\noindent
In order to deal with the reducible case, one needs to address the \emph{oriented} nature of the uniqueness part of the Kneser--Milnor decomposition.  Specifically, a proof of the following statement is required:

\begin{quote}
\emph{There is an algorithm that determines whether   a closed, orientable, irreducible $3$-manifold admits an orienta\-tion-re\-ver\-sing self-homeomorphism.}
\end{quote}
\noindent
In the case of a hyperbolic manifold $M$ of finite volume, this fact can be deduced from published results.  By Mostow Rigidity, any automorphism of~$\pi_1M$ is induced by  a self-homeomorphism.   The extension of  Sela's results by Dahmani--Groves \cite{dahmani_isomorphism_2008} makes it possible to list the elements $\{\alpha_i\}$ of the (finite) outer automorphism group of $\pi_1M$, and computing the action on~$H_3$ determines whether or not each $\alpha_i$ reverses orientation.  Alternatively, one can attempt to modify the Scott--Short algorithm~\cite{SS12} (cf.~\cite{Kupb}).\medskip

\noindent
In order to complete the solution to the Homeomorphism Problem in the closed case, one needs to analyze the remaining Seifert fibered and Haken cases.  Perhaps a variation of the Haken--Hemion--Matveev algorithm \cite[Theorem~6.1.1]{Mat03} can be used to solve the Oriented Homeomorphism Problem in the Haken case.  Alternatively, using the canonical properties of the JSJ decomposition, it should be possible to reduce the question to the hyperbolic and Seifert fibered cases.\medskip

\noindent
In our opinion, a detailed solution of the above problem would be a great service to the community, and fill an important gap in the literature.\medskip

\noindent
Another problem that has attracted a lot of attention is the problem of deciding whether a given $3$-manifold is homeomorphic to a particular kind of $3$-manifold (e.g.,\ spheres, handlebodies etc.).  For example,
Rubinstein~\cite{Ru95} gave an algorithm that determines whether a  $3$-manifold is homeomorphic to a $3$-sphere, whose  correctness was shown in \cite{Tho94}.
Later, Ivanov~\cite{Iv08} and Schleimer~\cite{Schl11} showed that this  problem is the complexity class~NP, that is, decidable in nondeterministic polynomial time (in size of the triangulation). In his paper~\cite{Iv08}, Ivanov also gave another proof of the result of Hass--Lagarias--Pippenger~\cite{HLP} that the problem of detecting 
whether a knot (represented by a triangulation of its complement) is
the unknot, which 
was first shown to be decidable by Haken~\cite{haken_normalflaechen}, is in NP.
Kuperberg~\cite{Kuperberg}, using a theorem of Koiran~\cite{Koiran}, has shown that modulo the Generalized Riemann Hypothesis,  unknot detection is
also in co-NP, so (assuming also standard conjectures in complexity theory) not NP-hard.
For some NP-complete decision problems in knot theory, see \cite{AHT}.

\section{Open Problems}\label{sec:open problems}

\noindent
We conclude this survey paper with a list of problems and conjectures.
As we saw in Theorem \ref{thm: IP}, the 
Isomorphism Problem for the class of fundamental groups of closed, orientable $3$-manifolds has an algorithmic solution. It is  natural to ask whether the restriction to closed $3$-manifolds is necessary.
 
\begin{question}
Is there an algorithm which determines whether two fundamental groups of $3$-manifolds are isomorphic?
\end{question}

\noindent
The Equation Problem asks for a solution to the problem whether any set of `equations' over a group has a solution. It generalizes both the Word Problem and the Conjugacy Problem, and has been solved for torsion-free hyperbolic groups by Makanin \cite{Mak82} and  Rips--Sela \cite{RS98}, for hyperbolic groups with torsion by Dahmani--Guirardel \cite{dahmani_foliations_2010} and for fundamental groups of Seifert fibered spaces by Liang \cite{Li14}.  The following question thus arises.

\begin{question}
Is the Equation Problem solvable for the fundamental group of any $3$-manifold?
\end{question} 

\noindent
Even more ambitiously, one may ask about the decidability of the full elementary theory of each $3$-manifold group viewed as a structure in the language of groups, in the sense of model theory; see \cite[Example~1.2.5]{marker}.

\begin{question}
Let $\pi$ be a $3$-manifold group.
Is the first-order theory of $\pi$ decidable?
\end{question}

\noindent
See \cite{km06, km13} for work on this problem  in the case where $\pi$ is a free group respectively a
torsion-free hyperbolic group, but see also \cite{km14,se14}. 
 


\medskip
\noindent
Finally, we return to topological decision problems. A {\it $3$-manifold pair}\/ is a pair $(N,S)$ where $N$ is a $3$-manifold and $S$ is a subsurface in $\partial N$. For example, let $L=L_1\cup\dots\cup L_n$ be an oriented link in a $3$-manifold $Y$. Pick a tubular neighborhood $\nu L$ in $Y$
and   denote by $\mu_1,\dots,\mu_n$ the meridians of $L_1,\dots,L_n$, respectively.
The diffeomorphism class of~$(Y,L)$ is then  determined by the 
diffeomorphism class of the $3$-manifold pair $(Y\setminus \nu L,\mu_1\times I\cup \dots\cup \mu_n\times I)$. Sutured manifolds give naturally rise to $3$-manifold pairs. 

\begin{question}
Is there a solution to the Homeomorphism Problem for $3$-manifold pairs?
\end{question}

\end{document}